\documentclass{article}
\usepackage{amsmath}
\usepackage{amscd}
\usepackage{amsfonts}
\usepackage{amssymb}
\usepackage{mathrsfs}
\usepackage{amsthm}

\theoremstyle{plain}
\newtheorem{theorem}{Theorem}
\newtheorem{proposition}[theorem]{Proposition}

\newcommand{\und}{\underline}
\newcommand{\Ocal}{\mathscr{O}}
\newcommand{\Lcal}{\mathscr{L}}
\newcommand{\Ical}{\mathscr{I}}
\newcommand{\Mcal}{\mathscr{M}}
\DeclareMathOperator{\Pic}{Pic}
\DeclareMathOperator{\Hom}{Hom}

\begin{document}
\title{Autoduality for curves of compact type }
\author{Eduardo Esteves and Fl\'avio Rocha}
\maketitle

\section{Introduction}

Let $C$ be a connected projective reduced curve defined 
over an algebraically closed field,
$J=\Pic^0(C)$ the connected component of the identity of its 
Picard scheme and $\mathscr{L}$ a line bundle of degree 1 over $C$. 
The classical 
\emph{Autoduality Theorem} says the following: 
If $C$ is smooth, the Abel map $A:C \to  {J}$, 
given by $P \mapsto  \mathscr{L} \otimes \mathscr{I}_P$ 
where $\mathscr{I}_P$ is the sheaf of ideals of $P$, is well 
defined and induces an isomorphism 
$A^*: \textrm{Pic}^0(J) \to  J$ which does not depend on 
the choice of $\mathscr{L}$; see \cite{mudd}, Prop.~6.9, p.~118. 
We say that $J$, the \emph{Jacobian} of $C$, is \emph{autodual}.

In a more general setup, where $C$ has singularities, 
the Autoduality Theorem was first proved in \cite{e} for 
irreducible curves with at most double points, 
with $J$ being replaced by its natural compactification $\bar{J}$, the moduli 
space of degree-0 torsion-free rank-1 sheaves over $C$, 
constructed by D'Souza \cite{ds} and Altman and Kleiman \cite{ak}. 
Later, Arinkin \cite{ar} 
extended the validity of the theorem to irreducible curves with at 
most planar singularities. (Also, the autoduality isomorphism extends to 
compactifications, as proved in \cite{ek} and \cite{ar2}.)

The compactified Jacobian $\bar J$ appears as a singular fiber 
of the Hitchin fibration \cite{hi}. More precisely, 
the Hitchin fibration is a map from the moduli 
space of Higgs bundles to that of spectral curves whose fiber over a 
given spectral curve parameterizes torsion-free sheaves over that curve. 
Autoduality of compactified Jacobians is thus related to autoduality 
of the Hitchin fibration; see discussion in \cite{ar}, Sect.~7.

However, spectral curves need not be irreducible. So a natural question 
arises: Does autoduality hold for reducible curves, like 
Deligne--Mumford stable ones? In this note we show autoduality for curves 
of compact type and, more generally, treelike curves with 
planar singularities. 

Compactified Jacobians of reducible curves are less understood and more 
complex. For instance, they may contain 
more than one copy of $J$. Nevertheless, an Abel map $A:C\to\bar J$ 
has been constructed when $C$ is Gorenstein; see \cite{l}, which extends 
\cite{ce}. In case $C$ is of compact type, the image of $A$ lies inside a 
copy of $J$, as observed and studied in \cite{cp}. 

So, assume from now on that $C$ is of compact type, with irreducible 
components $C_1,...,C_n$. For each $n$-tuple 
${\und d}=(d_1,...,d_n) \in \mathbb{Z}^n $, let 
$J^{\und d}$ be the scheme parameterizing line bundles over $C$ of 
multidegree $\und d$. Thus $J=J^{\und 0}$ and each $J^{\und d}$ is 
isomorphic to $J$. For each $i=1,\dots,n$ and each integer $e$, let 
$J^e_i$ be the corresponding scheme for $C_i$. 
Restriction to the components of $C$ gives us a 
natural map 
$$
\pi^{\und d}=(\pi_1^{\und d},\dots,\pi_n^{\und d}):
J^{\und d}\to J^{d_1}_1\times\cdots\times J^{d_n}_n,
$$
which is an isomorphism because $C$ is of compact type.

Abel maps come in all sorts, but there is a common 
property of those constructed so far that we explore. We say that a map 
$A:C\to J^{\und d}$ is a \emph{decomposable Abel map} 
if $\pi_j^{\und d}A|_{C_i}$ is constant for $i\neq j$ and 
$A_i:=\pi_i^{\und d}A|_{C_i}$ is an Abel map for 
each integer $i=1,\dots,n$, that is, there is a 
line bundle $\mathscr{L}_i$ over $C_i$ 
such that $A_i$ sends $P$ 
to $\mathscr{L}_i\otimes\Ocal_{C_i}(-P)$. 
We say that $(\Lcal_1,\dots,\Lcal_n)$ describes $A$. 

Finally, we prove the Autoduality Theorem: 
\emph{If $C$ is a curve of compact type, 
and $A:C\to J^{\und d}$ is a decomposable Abel map, the pullback 
$A^*:\Pic^0(J^{\und d})\to  J$ is an isomorphism which independs 
of the $n$-tuple $(\Lcal_1,\dots,\Lcal_n)$ describing $A$.}

The proof relies on observing that the various pullbacks form a commutative 
diagram
$$
\begin{CD}
\textrm{Pic}^0(J^{\und{d}}) @>A^*>> J\\
@VVV @V\pi^{\und 0}VV\\
\Pic^0(J^{d_1}_1)\times\cdots\times\Pic^0(J^{d_n}_n) 
@>(A_1^*,\dots,A_n^*)>> J^0_1\times \cdots \times J^0_n 
\end{CD}
$$
where $A_i:=\pi_i^{\und d}\circ A|_{C_i}$ for $i=1,\dots,n$, and the left 
vertical map is induced by the isomorphism $\pi^{\und d}$; see Proposition 1. 
Both vertical maps are isomorphisms. By the classical Autoduality 
Theorem, so is the bottom map. Hence, the commutativity of the diagram yields 
that $A^*$ is also an isomorphism.

The same statement and proof hold if each $A_i$ maps $P$ to 
$\mathscr{L}_i\otimes\Ocal_{C_i}(P)$. Also, $C$ need not be of compact type: 
We need only that each component $C_i$ has at most planar singularities and 
intersects its complement in $C$ at separating nodes. In this case, a 
similar isomorphism to $\pi^{\und d}$ exists and Arinkin's Autoduality 
Theorem can be used for the components. We give the more general statement 
in Theorem 2 below.

A generalization of the result obtained in this note has been recently
made available at \cite{MRV} by Melo, Rapagnetta and Viviani, with an
appendix by L\'opez-Mart\'\i n. They claim autoduality for any curve
with planar singularities.

\section{Autoduality}

Let $C$ be a connected projective reduced scheme of dimension $1$ over an 
algebraically closed field $k$, 
a \emph{curve} for short. A \emph{subcurve} of $C$ 
is a reduced union of irreducible components of $C$, thus not necessarily 
connected. A point $P\in C$ is a \emph{separating node} if $P$ is an ordinary 
node of $C$ and $C-P$ is not connected. A point $P\in C$ is called 
\emph{reducible} if it lies on at least two irreducible components of $C$. We 
say that $C$ is \emph{treelike} if all its reducible points are separating 
nodes.

A coherent sheaf $\Ical$ over $C$ is said to be \emph{torsion-free} if it is 
isomorphic to a subsheaf of the (constant) sheaf of rational functions, 
\emph{rank-$1$} if it is generically invertible, and \emph{simple} if 
$\Hom(\Ical,\Ical)=k$. The \emph{degree} of a torsion-free, rank-1 sheaf 
$\Ical$ is $\deg(\Ical):=\chi(\Ical)-\chi(\Ocal_C)$. 
It follows from \cite{est}, Prop.~1, p.~3049, 
that a torsion-free rank-1 sheaf over $C$ is simple only if it is 
invertible at separating points. The converse is true if $C$ is treelike, in 
which case the restriction of a 
simple torsion-free rank-1 sheaf to any connected subcurve of $C$ is also 
simple, torsion-free and rank-1.

\emph{Assume from now on that $C$ is tree-like.} 
Let $C_1,\dots,C_n$ denote the irreducible components of $C$. 
For each $n$-tuple ${\und d}=(d_1,...,d_n) \in \mathbb{Z}^n $, let 
$\bar J^{\und d}$ be the scheme parameterizing simple torsion-free rank-1 
sheaves $\Ical$ over $C$ such that $\deg(\Ical|_{C_i})=d_i$ for 
$i=1,\dots,n$. For each $i=1,\dots,n$ and each integer $e$, let 
$\bar J^e_i$ be the corresponding scheme 
for $C_i$. Restriction to the components of $C$ gives us a 
natural map: 
$$
\bar\pi^{\und d}=(\bar\pi_1^{\und d},\dots,\bar\pi_n^{\und d}):
\bar J^{\und d}\longrightarrow\bar J^{d_1}_1\times\cdots\times\bar J^{d_n}_n.
$$
It follows from \cite{e1}, Prop.~3.2, p.~172, that $\bar\pi^{\und d}$ is an 
isomorphism.

For each ${\und d}=(d_1,...,d_n) \in \mathbb{Z}^n $, let 
$J^{\und d}\subseteq\bar J^{\und d}$ be the open subscheme parameterizing 
invertible sheaves. Likewise, for each $i=1,\dots,n$ and each integer $e$, let 
$J^e_i\subseteq\bar J^e_i$ be the open subscheme parameterizing 
invertible sheaves. Then $\bar\pi^{\und d}$ restricts to an isomorphism
$$
\pi^{\und d}=(\pi_1^{\und d},\dots,\pi_n^{\und d}):
J^{\und d}\longrightarrow J^{d_1}_1\times\cdots\times J^{d_n}_n.
$$

\emph{Assume from now on that the singularities of $C$ are planar.} Then so 
are the singularities of $C_i$ for each $i$. It follows from 
\cite{aik}, (9), p.~8, that $\bar J^e_i$ is integral for each $e$. Thus, 
by \cite{groth}, Thm.~3.1, p.~232-06, there is a scheme parameterizing line 
bundles over $\bar J^e_i$, whose connected component of the identity 
we denote by $\Pic^0(\bar J^e_i)$. Similarly, since 
$\bar\pi^{\und d}$ is an isomorphism, $\bar J^{\und d}$ is integral for each 
$\und d$, and there is a corresponding scheme for $\bar J^{\und d}$, whose 
connected component of the identity we denote by $\Pic^0(\bar J^{\und d})$.

\begin{proposition} Let $\und d\in\mathbb Z^n$. For each $j=1,\dots,n$, let 
$\Ical_j$ be a degree-$d_j$ torsion-free rank-$1$ sheaf over $C_j$, and let 
$$
\iota_j:\bar J^{d_j}_j\longrightarrow\bar J^{\und d}
$$ 
be the composition of the inverse of $\bar\pi^{\und d}$ with the map given by 
$$
\Ical\mapsto(\Ical_1,\dots,\Ical_{j-1},\Ical,\Ical_{j+1},\dots,\Ical_n).
$$
Then the induced map
$$
\iota^*:=(\iota_1^*,\dots,\iota_n^*):\Pic^0(\bar J^{\und d})\longrightarrow
\Pic^0(\bar J^{d_1}_1)\times\dots\times\Pic^0(\bar J^{d_n}_n)
$$ 
is an isomorphism. Furthermore, if $\Lcal_j$ is a degree-$0$ invertible 
sheaf over $C_j$ for $j=1,\dots,n$, then replacing each $\Ical_j$ by 
$\Ical_j\otimes\Lcal_j$ does not change $\iota^*$.
\end{proposition}

\begin{proof} Since $\bar\pi^{\und d}$ is an isomorphism, and since the 
$\bar J^{d_j}_j$ are complete varieties, the proof of the first statement is 
a simple application of the \emph{Theorem of the Cube}, in an
extended version: 
\emph{Let 
$X_1,\dots,X_n$ be complete varieties and $P_1,\dots,P_n$ points on each of 
them. 
Set $X:=X_1\times\dots\times X_n$. For each $j=1,\dots,n$, let 
$\phi_j:X_j\to X$ be the map taking $P$ to 
$(P_1,\dots,P_{j-1},P,P_{j+1},\dots,P_n)$. Then 
$$
\phi^*:=(\phi_1^*,\dots,\phi_n^*):\Pic^0(X)\longrightarrow
\Pic^0(X_1)\times\dots\times\Pic^0(X_n)
$$
is an isomorphism.} 

However, 
not being able to find a proof of the above statement in the literature, 
we give its proof. Consider the map 
$$
\psi:\Pic^0(X_1)\times\dots\times\Pic^0(X_n)\longrightarrow\Pic^0(X)
$$
given by taking $(\Lcal_1,\dots,\Lcal_n)$ to 
$\Lcal_1\boxtimes\cdots\boxtimes\Lcal_n$. It is clear that $\phi^*\psi=1$. We 
need only show $\psi\phi^*=1$ as well.

Let $Y:=\Pic^0(X)$ and $Q\in Y$ be the point corresponding to the trivial line 
bundle. Let $\Lcal$ denote the universal line bundle on 
$X\times Y$ rigidified 
along $P\times Y$, where $P:=(P_1,\dots,P_n)$. Thus $\Lcal|_{P\times Y}$ and 
$\Lcal|_{X\times Q}$ are trivial. For $j=1,\dots,n$, let 
$\lambda_j:=(\phi_j,1_Y)$ and 
$\rho_j:X\times Y\to X_j\times Y$ be the projection map. 
Then $\psi\phi^*$ is induced by 
$\rho_1^*\lambda_1^*\Lcal\otimes\cdots\otimes\rho_n^*\lambda_n^*\Lcal$. 

Let $\Mcal:=\rho_1^*\lambda_1^*\Lcal\otimes\cdots\otimes\rho_n^*\lambda_n^*\Lcal
\otimes\Lcal^{-1}$. Then $\Mcal$ is a line bundle on $X\times Y$ such that 
$\lambda_j^*\Mcal$ is trivial for each $j=1,\dots,n$ and $\Mcal|_{X\times Q}$ is 
trivial. We need only show that any such line bundle is trivial.

The case $n=1$ is trivial, while the case $n=2$ is a direct application of 
the Theorem of the Cube, as stated in \cite{muav}, p.~91. For the general 
case, we may assume by induction that the restriction of $\Mcal$ 
to $X_1\times\dots\times X_{n-1}\times P_n\times Y$ is trivial. Then, 
applying the Theorem of the Cube to the varieties 
$X_1\times\dots\times X_{n-1}$, $X_n$ and $Y$ finishes the proof of the 
first statement.

As for the second statement, 
for each $j=1,\dots,n$ consider the translation 
$\tau_{\Lcal_j}:\bar J^{d_j}_j\to\bar J^{d_j}_j$, sending $\Ical$ to 
$\Ical\otimes\Lcal_j$. It is enough to prove that the induced map 
$\tau_{\Lcal_j}^*$ on $\Pic^0(\bar J^{d_j}_j)$ is the identity. This is proved 
in \cite{ek}, Prop.~2.5, p.~489, for curves whose singularities are double 
points. However, even if $C_j$ has worst singularities, but planar, 
an analysis of the proof given there shows first that 
we may assume, in any case, that $d_j=0$. 

Second, fix any invertible sheaf $\Mcal$ on $C_j$ of 
degree 1 and consider the corresponding Abel maps $A_1,A_2:C\to\bar J^0_j$, 
the first defined by $P\mapsto\Mcal\otimes\Ical_P$, the second 
by $P\mapsto\Mcal\otimes\Lcal_j\otimes\Ical_P$. Clearly, 
$A_2=\tau_{\Lcal_j}A_1$. By \cite{e}, Prop.~3.7, p.~605, the pullback maps 
$A_1^*,A_2^*:\Pic^0(\bar J^0_j)\to J^0_j$ are equal. Thus 
$A_1^*=A_1^*\tau_{\Lcal_j}^*$. 

Finally, it follows from \cite{ar}, Thm.~C, 
that a certain map $\rho:J^0_j\to\Pic^0(\bar J^0_j)$ is an isomorphism. This 
map is the one called $\beta$ in \cite{e}, Prop.~2.2, p.~595, where it is 
proved that $A_1^*\beta=1$. So $A_1^*$ is an isomorphism as well. Since 
$A_1^*=A_1^*\tau_{\Lcal_j}^*$, it follows that $\tau_{\Lcal_j}^*$ is the 
identity.
\end{proof}

Let $A: C\to\bar J^{\und d}$ be a map. If $C$ is irreducible, we say that 
$A$ is an \emph{Abel map} if 
there is an invertible sheaf $\Lcal$ over $C$ such 
that $A$ sends $P$ to $\Lcal\otimes\Ical_P$ for each $P\in C$. More 
generally, we say that $A$ is a 
\emph{decomposable Abel map} 
if $\bar\pi_j^{\und d}A|_{C_i}$ is constant for $i\neq j$ and 
$A_i:=\bar\pi_i^{\und d}A|_{C_i}$ is an Abel map for 
each integer $i=1,\dots,n$, that is, there is an invertible 
sheaf $\mathscr{L}_i$ over $C_i$ 
such that $A_i$ sends $P$ 
to $\mathscr{L}_i\otimes\Ical_P$ 
for each $P\in C_i$. We say that $(\Lcal_1,\dots,\Lcal_n)$ \emph{describes} 
$A$. 

Indeed, for $i\neq j$ let $Y_{i,j}$ be the 
connected component of $\overline{C-C_j}$ 
containing $C_i$. Since $C$ is treelike, $Y_{i,j}$ meets $C_j$ at a unique 
point $N_{i,j}$. Then $\bar\pi_j^{\und d}A|_{C_i}$ has constant image 
$\Lcal_j\otimes\Ocal_{C_j}(-N_{i,j})$. 

Conversely, given invertible sheaves 
$\Lcal_1,\dots,\Lcal_n$ on $C_1,\dots,C_n$ 
of degrees $d_1+1,\dots,d_n+1$, there is a decomposable Abel map 
$A: C\to\bar J^{\und d}$ described by $(\Lcal_1,\dots,\Lcal_n)$: For each 
$P\in C$ and each $j=1,\dots,n$, the map $\bar\pi_j^{\und d}A$ sends $P$ to 
$\Lcal_j\otimes\Ocal_{C_j}(-P)$ if $P\in C_j$, and to 
$\Lcal_j\otimes\Ocal_{C_j}(-N)$ if $P\not\in C_j$, where $N$ is the point of 
intersection with $C_j$ of the connected component of $\overline{C-C_j}$ 
containing $P$.

The Abel maps constructed in \cite{l} are decomposable, as it follows from 
\cite{l}, Lemma 3, p.~46.

\begin{theorem} Let $C$ be a treelike curve whose singularities are planar. 
Let $A: C\to\bar J^{\und d}$ be the decomposable Abel map described 
by $(\Lcal_1,\dots,\Lcal_n)$. Then the induced map
$$
A^*:\Pic^0(\bar J^{\und d})\longrightarrow\Pic^0(C)
$$
is an isomorphism which does not depend on the choice of 
$(\Lcal_1,\dots,\Lcal_n)$
\end{theorem}

\begin{proof} For $j\neq i$, let $N_{j,i}$ be the point of 
intersection of $C_i$ with the connected component of $\overline{C-C_i}$ 
containing $C_j$. For each $j=1,\dots,n$, let $A_j:=\pi_j^{\und d}A|_{C_j}$ and 
define $\iota_j:\bar J^{d_j}_j\to\bar J^{\und d}$ as 
the composition of the inverse of 
$\bar\pi^{\und d}$ with the map given by 
$$
\Ical\mapsto(\Lcal_1(-N_{j,1}),\dots,\Lcal_{j-1}(-N_{j,j-1}),\Ical,
\Lcal_{j+1}(-N_{j,j+1}),\dots,\Lcal_n(-N_{j,n})),
$$
where $\Lcal_i(-N_{j,i}):=\Lcal_i\otimes\Ocal_{C_j}(-N_{j,i})$ for 
$i\neq j$. Then 
$$
\iota_jA_j=A|_{C_j}\quad\text{for $j=1,\dots,n$.}
$$

Thus we obtain the 
commutativity of the diagram of maps
$$
\begin{CD}
\textrm{Pic}^0(\bar J^{\und{d}}) @>A^*>> J\\
@V\iota^*VV @V\pi^{\und 0}VV\\
\Pic^0(\bar J^{d_1}_1)\times\cdots\times\Pic^0(\bar J^{d_n}_n) 
@>(A_1^*,\dots,A_n^*)>> J^0_1\times \cdots \times J^0_n 
\end{CD}
$$
where $\iota^*:=(\iota_1^*,\dots,\iota_n^*)$. It follows from Statements 
1 and 2 of Proposition 1 that $\iota^*$ is an isomorphism. Furthermore, 
since $A_j$ is an Abel map, it follows from \cite{ar}, Thm.~C, and 
\cite{e}, Prop.~2.2, p.~595, as in the proof of Proposition 1, 
that $A_j^*$ is an 
isomorphism for each $j=1,\dots,n$. Finally, $\pi^{\und 0}$ is an 
isomorphism because $C$ is treelike. Then the 
commutativity of the diagram above yields that $A^*$ is an isomorphism.

That $A^*$ does not depend on the choice of $(\Lcal_1,\dots,\Lcal_n)$ follows 
from the independence of $\iota^*$, stated in Proposition 1, and the 
independence of the $A_i^*$, stated in \cite{e}, Prop.~3.7, p.~605.
\end{proof}

\end{document}